\documentclass[amsfonts, a4paper,11pt]{amsart}

\usepackage{
  amsmath
 ,amssymb
 ,amsfonts
 ,faktor
 ,stmaryrd
 ,tikz-cd
 ,bbold
 ,enumerate}

\newenvironment{enumroman}{\begin{enumerate}[\upshape (i)]}{\end{enumerate}}

\usepackage[all,cmtip]{xy}

\theoremstyle{plain}
	\newtheorem{thm}{Theorem}[subsection] 
	\newtheorem{cor}[thm]{Corollary}
	\newtheorem{prop}[thm]{Proposition}
	\newtheorem{lem}[thm]{Lemma}

\theoremstyle{definition}
	\newtheorem{defn}[thm]{Definition}
	\newtheorem{ex}[thm]{Example}
	\newtheorem{rmk}[thm]{Remark}

\def\class#1{\mathcal{#1}}
	\def\A{\class{A}}
	\def\B{\class{B}}
	\def\C{\class{C}}
	
	\def\E{\class{E}}
	\def\K{\class{K}}

	\def\V{\class{V}}
	
\def\2{\mathbb{2}}

\renewcommand{\textbf}[1]{\text{\fontseries{b}\selectfont{\upshape #1}}}
\def\cate#1{\textbf{#1}}
	\def\Cat{\cate{Cat}}
	\def\Set{\cate{Set}}

	\def\Mod{\cate{Mod}}
	\def\Ring{\cate{Ring}}
	\def\Fib{\cate{Fib}}
	\def\DFib{\cate{DFib}}
	\def\DCof{\cate{DCof}}
	\def\Span{\cate{Span}}
	\def\Opspan{\cate{Opspan}}

\def\a{\alpha}

\DeclareMathOperator{\ob}{Ob}

\DeclareMathOperator{\cod}{cod}

\def\op{\mathrm{op}}
\def\co{\text{co}}

\def\id{\text{id}}

\def\ra{\rightarrow}
\def\la{\leftarrow}

\def\sr{\stackrel}
\def\Ra{\Rightarrow}

\usepackage{tikz}
\usetikzlibrary{fit,shapes}

\pgfdeclarelayer{bg}    % declare background layer
\pgfsetlayers{bg,main}  % set the order of the layers (main is the standard layer)

\newcommand{\eadd}[1]{#1}

\newcommand{\comma}[2]{(#1/#2)}

\usepackage{hyperref}

\hypersetup{%
  pdftoolbar=   true,
  pdfmenubar=   true,
  pdffitwindow= true,
  pdftitle=     {Categorical notions of fibration},
  pdfauthor=    {Fosco Loregian \& Emily Riehl},
  colorlinks=   true,
  linkcolor=    black,
  citecolor=    red!70!black}

\begin{document}

\title{Categorical notions of fibration}

\author{Fosco Loregian}
\address{
{\sf Fosco Loregian} : 
Max Planck Institute for Mathematics, 
Vivatsgasse 7, 53111 Bonn --- Germany, 
\href{mailto:flore@mpim-bonn.mpg.de}{\sf flore@mpim-bonn.mpg.de}
}

\author{Emily Riehl}
\address{
{\sf Emily Riehl} : 
Department of Mathematics \\
Johns Hopkins University\\
3400 N Charles Street\\
Baltimore, MD 21218}
\email{\href{mailto:eriehl@math.jhu.edu}{\sf eriehl@math.jhu.edu}}

\date{Original version December 20, 2010; revised version \today}

\begin{abstract} 
\eadd{Fibrations over a category $B$, introduced to category theory by Grothendieck, encode pseudo-functors $B^\op \rightsquigarrow \Cat$, while the special case of discrete fibrations encode presheaves $B^\op \to \Set$. A two-sided discrete variation encodes functors $B^\op \times A \ra \Set$, which are also known as profunctors from $A$ to $B$. By work of Street, all of these fibration notions can be defined internally to an arbitrary 2-category or bicategory. While the two-sided discrete fibrations model profunctors internally to $\Cat$, unexpectedly, the dual two-sided codiscrete cofibrations are necessary to model $\V$-profunctors internally to $\V$-$\Cat$.} These notes were initially written by the second-named author to accompany a talk given in the Algebraic Topology and Category Theory Proseminar in the fall of 2010 at the University of Chicago. A few years later, the  first-named author joined to expand and improve the internal exposition and external references.
\end{abstract}

\maketitle
\setcounter{tocdepth}{1}
\tableofcontents 
\section{Introduction}

Fibrations were introduced to category theory in \cite{grothendieckcategoriesfibrees,grothendieckdescenteI} and developed in \cite{grayfibred}. Ross Street gave definitions of fibrations internal to an arbitrary 2-category \cite{streetfibrationsyoneda} and later bicategory \cite{streetfibrationsbicategories}. Interpreted in the 2-category of categories, the 2-categorical definitions agree with the classical ones, while the bicategorical definitions are more general.

In this expository article, we tour the various categorical notions of fibration in order of increasing complexity. We begin in Section \ref{1catsec} with the classical definitions of fibrations and discrete fibrations in ordinary 1-category theory. The internalization in a 2-category and generalization in a bicategory are given in Sections \ref{2catsec} and \ref{bicatsec}. The real goal, which we pursue in parallel, is to define two-sided discrete fibrations. In $\Cat$, two-sided discrete fibrations encode functors $B^\op \times A \ra \Set$ also known as \emph{profunctors} from $A$ to $B$, while in $\V$-$\Cat$ the dual two-sided codiscrete cofibrations encode $\V$-profunctors $\B^\op \otimes \A \ra \V$. We conclude with a construction of a bicategory, defined internally to $\V$-$\Cat$, whose 1-cells are two-sided codiscrete cofibrations. This bicategory can be used to ``equip'' the 2-category of $\V$-categories with a bicategory of $\V$-profunctors, providing a fertile setting for the formal category theory of enriched categories.

This theory has been extended to $(\infty,1)$-categories modeled as quasi-categories by Andr\'e Joyal and Jacob Lurie. In that context, the equivalence between fibrations and pseudofunctors is implemented by the \emph{straightening} and \emph{unstraightening} constructions of Lurie, plays a particularly important role. Because the foundational $(\infty,1)$-category theory required to present these homotopical variants of categorical fibrations is quite extensive, we regretfully do not include this material here and refer the reader instead to \cite{lurietopos} and \cite{riehlverityelements}.

In our attempt to cover a lot of material as expediently as possible, we give only a few sketched proofs but take care to provide thorough citations. There are many categorical prerequisites, particularly in the later sections, but we believe they are strictly easier than the topics below that take advantage of them.

\subsection*{Comma categories}
One categorical prerequisite is so important to merit a brief review. Given a pair of functors $B \sr{f}{\ra} C \sr{g}{\la} A$, an \emph{opspan} in $\Cat$, the \emph{comma category} $\comma{f}{g}$ has triples $(b \in B, fb \ra ga, a \in A)$ as objects and morphisms $(b,fb \ra ga, a) \ra (b', fb' \ra ga', a')$ given by a pair of arrows $a \ra a' \in A$, $b \ra b' \in B$ such that the evident square commutes in $C$. This category is equipped with canonical projections to $A$ and $B$ as well as a 2-cell
\[
\xymatrix@=10pt{ & \comma{f}{g} \ar[dl]_c \ar[dr]^d \ar@{}[dd]|{\Leftarrow} \\ A \ar[dr]_g & & B \ar[dl]^f \\ & C}
\]
and is universal among such data. 

Equivalently, let us denote $\Lambda^2_2$ the \emph{generic opspan}, i.e.~the category $\{0\ra 2\la 1\}$. Then $\comma{f}{g}$ is the limit of $\Lambda^2_2\sr{F}{\mapsto} \{B \sr{f}{\ra} C \sr{g}{\la} A\}$ weighted by $\Lambda^2_2 \sr{W}{\mapsto} \{\mathbb{1} \sr{d}{\ra} \2 \sr{c}{\la} \mathbb{1}\}$, the inclusions of the terminal category as the domain and codomain of the walking arrow. Equivalently --- as $\Cat$-weighted limits can be written as ends, in this case the end $\int_{i:\Lambda^2_2} \{Wi,Fi\}$ --- $\comma{f}{g}$ is the equalizer
\[
\xymatrix{
\comma{f}{g}\ar[r] & \displaystyle \prod_{i : \Lambda^2_2} \{Wi, Fi\} \ar@<4pt>[r]\ar@<-4pt>[r] & \displaystyle \prod_{(i\to j) : \Lambda^2_2} \{Wi, Fj\}
}
\]
(see \cite{kelly1989elementary} for the notation). Unwinding the definitions, this is precisely the limit of the diagram of categories
\[
\xymatrix@=10pt{
	& A \ar[dr]^g\ar@{=}[dl]&& C^{\2} \ar[dr]^d\ar[dl]_c&& B\ar@{=}[dr]\ar[dl]_f \\
	A && C && C && B.
}
\]
Often, we are interested in comma categories in which either $f$ or $g$ is an identity (in which case it is denoted by the name of the category) or in which either $A$ or $B$ is terminal (in which case the functor is denoted by the object it identifies). Such categories are sometimes called \emph{slice categories}.

\subsection*{A few notes on terminology} What we call \emph{fibrations} in $\Cat$ are sometimes called \emph{categorical}, \emph{Grothendieck}, \emph{cartesian}, or \emph{right} (and unfortunately also \emph{left}) \emph{fibrations}. The left-handed version, now \emph{opfibrations}, was originally called \emph{cofibrations}, though this name was rejected to avoid confusing topologists. Somewhat unfortunately, as we shall see below, once fibrations have been defined internally to a 2-category $\K$, the \emph{opfibrations} are precisely the fibrations in $\K^\co$ (formed by reversing the 2-cells only), while the \emph{cofibrations} are precisely the fibrations in $\K^\op$ (formed by reversing the 1-cells only). 

\subsection*{Acknowledgments} In addition to the cited sources, the second author learned about this material from conversations with Dominic Verity and Mike Shulman, and also from Urs Schreiber and the $n$Lab, a wiki devoted to category theory and higher category theory. She was supported by the NSF Graduate Research Fellowship Program when these notes were first prepared and by an NSF CAREER grant DMS-1652600 during the revision.

The first author thanks the second author for having permitted him to join the project and for having let him tinker with the {\TeX} source more than usual. 
\section{Fibrations in 1-category theory}\label{1catsec}

Loosely, a \emph{fibration} is a functor $p\colon E \ra B$ such that the fibers $E_b$ depend contravariantly and pseudo-functorially on the objects $b \in B$. Many categories are naturally \emph{fibered} in this way. 

\subsection{Discrete fibrations} We start with an easier variant  where the fibers are discrete categories. 

\begin{defn} A functor $p \colon E \ra B$ is a \emph{discrete fibration} if for each object $e \in E$ and arrow $f\colon b' \ra pe \in B$, there exists a unique lift $g\colon e' \ra e$.
\end{defn}

\eadd{In particular, for each $b \in B$ and object $e \in E_b$ in the fiber, there can be at most one lift of $1_b$ with codomain $e$. This proves that the fibers $E_b$ of a discrete fibration are discrete categories.}

Let $\DFib(B)$ denote the category of discrete fibrations over $B$, defined to be a full subcategory of the comma category $\Cat_{/B}$. Two facts about discrete fibrations are particularly important, the first of which we explore now.

\begin{thm}\label{thm:discrete-equiv} There is an equivalence of categories \[\DFib(B) \simeq [B^{\op},\Set].\]
\end{thm}
\begin{proof} 
Given a discrete fibration $E \ra B$, define $B^{\op} \ra \Set$ by $b \mapsto E_b$, the category whose objects sit over $b \in B$ and whose arrows map to the identity at $b$. %Because $1_b$ has a unique lift for each $e \in E_b$, this category is discrete.
 For each morphism $f\colon b' \ra b$, define $f^* \colon E_b \ra E_{b'}$ by mapping $e \in E_b$ to the domain of the unique lift of $f$ with codomain $e$. Functoriality follows from uniqueness of lifts.

Conversely, given a functor $F\colon B^{\op} \ra \Set$, the canonical functor from its \emph{category of elements} $\comma{*}{F}$ to $B$ is a discrete fibration.  The category of elements construction is defined by applying the comma construction to opspans whose domain leg is fixed at $*\colon \mathbb{1} \to \Set$. In this way, the comma construction defines a functor $\comma{*}{-}\colon [B^\op, \Set] \to \Cat_{/B}$ whose value at $F \in [B^\op,\Set]$ is  a discrete fibration $\Sigma \colon \comma{*}{F}\to B$.

The comma category $\comma{*}{F}$, whose objects are elements of $Fb$ for some $b \in B$, is sometimes called the \emph{Grothendieck construction} on the presheaf $F$. We leave the verification that these functors define inverse equivalences to the reader.
\end{proof}

There is an ``internal'' rephrasing of the definition of a discrete fibration. Write $B_0$ for the set of objects and $B_1$ for the set of arrows of a small category $B$.

\begin{defn}\label{internaldefn} A functor $p : E \ra B$ between small categories is a \emph{discrete fibration} iff \[ \xymatrix{ E_1 \ar[d]_{p_1}  \ar@{}[dr]|(.2){\lrcorner}\ar[r]^{\cod} & E_0 \ar[d]^{p_0} \\ B_1 \ar[r]_{\cod} & B_0}\] is a pullback in $\Set$.
\end{defn}

 The second key property of discrete fibrations requires the following definition.

\begin{defn}
A functor $f \colon C \ra D$ is \emph{final} if any diagram of shape $D$ can be restricted along $f$ to a diagram of shape $C$ without changing its colimit. This is the case just when  for all $d \in D$, the comma category $\comma{d}{f}$ is non-empty and connected.
\end{defn}

\begin{prop}[\protect{\cite{MR0346027}}]\label{thm:dfib-ofs} There is an orthogonal factorization system $(\mathcal{E},\mathcal{M})$ on $\Cat$ with $\mathcal{E}$ the final functors and $\mathcal{M}$ the discrete fibrations.\footnote{This is usually called the \emph{comprehensive factorization system}. In \cite{lawvere1970equality} Lawvere defines the notion of a \emph{comprehension scheme} as a pseudofunctor $P \colon \mathcal E \to \text{Adj} $, sending a 1-cell $f \colon A\to B$ to an adjoint pair $Pf_! \colon PA\leftrightarrows PB \colon Pf^*$. Under suitable assumptions, such a $P$ defines a factorization system in $\mathcal E$ formed by classes of arrows called the $P$-\emph{connected} and $P$-\emph{covering} maps. When $\mathcal E = \Cat$ and $P$ is the presheaf construction $A\mapsto [A^\text{op}, \Set]$, a functor is $P$-connected if and only if it is final, and it is a $P$-covering if and only if it is a discrete fibration. See \cite{berger2017comprehensive} for a generalization to any `consistent' comprehension scheme.}
\end{prop}

The dual notion, functors $p \colon E \ra B$ which have unique lifts of morphisms with specified domain, are called \emph{discrete opfibrations}, which coincide with discrete fibrations $p \colon E^{\op} \ra B^{\op}$. These correspond  to functors $B \ra \Set$ and form an orthogonal factorization system with the class of \emph{initial} functors which are those such that restriction preserves limits.

\subsection{Fibrations} Now we're ready to introduce the main subject of our exposition.%for the real thing.
\begin{defn}\label{defn:1-fib} Given a functor $p \colon E \ra B$, an arrow $g \colon e' \ra e$ in $E$ is $p$-\emph{cartesian} if for any $g'\colon e'' \ra e$ such that $pg' = pg \cdot h$ in $B$, there is a unique lift $k$ of $h$ such that $g' = g \cdot k$. A functor $p :E \ra B$ is a \emph{fibration} if each $f: b \ra pe$ in $B$ has a $p$-cartesian lift with codomain $e$. 
\end{defn}

Some sources differentiate between fibrations and those \emph{cloven} fibrations which come with chosen cartesian lifts which may satisfy additional properties. See \cite{grayfibred} or the $n$Lab.

\begin{thm}\label{catadjthm} A functor $p \colon E \ra B$ is a fibration if and only if either 
\begin{enumroman} 
\item For each $e \in E$, the functor $p\colon \comma{E}{e} \ra \comma{B}{p(e)}$ has a right adjoint right inverse. 
\item The canonical functor $E^\2 \ra \comma{B}{p}$ obtained applying $p$ has a right adjoint right inverse.
\end{enumroman}
\end{thm}
\begin{proof} In each case, the right adjoint picks out $p$-cartesian lifts for each morphism. See \cite[Prop 3.11]{grayfibred}.
\end{proof}

Let $\Fib(B)$ denote the sub 2-category of $\Cat_{/B}$ of fibrations; \emph{cartesian functors}, those functors that preserve cartesian arrows; and all 2-cells.

\begin{thm} There is a 2-equivalence of 2-categories \[\Fib(B) \simeq {\bf Psd}[B^{\op},\Cat],\] where the latter is the 2-category of pseudo-functors $B^\op \rightsquigarrow \Cat$, pseudo-natural transformations, and modifications.
\end{thm}
\begin{proof} 
A fibration $E \ra B$ gives rise to a pseudo-functor $b \mapsto E_b$, $f \colon b' \ra b \mapsto f^*\colon E_b \ra E_b'$. By the universal property of the cartesian lifts, this assignment is functorial up to natural isomorphism. In the other direction, the lax colimit of a pseudo-functor $B^{\op} \rightsquigarrow \Cat$ is canonically a fibration over $B$. The fibration so-produced is frequently called the \emph{Grothendieck construction} of the associated pseudo-functor; see \cite[A1.1.7, B1.3.1]{MR1953060} for the discrete case, when the pseudo-functor takes values in $\Set$, and \cite[pp. 1-104]{MR2222646} for applications to algebraic geometry, which provided the historical source of this definition. 
\end{proof}

\begin{rmk}[Fibrations and indexed categories]
Let $\mathcal S$ be a category with finite limits, and in particular pullbacks. An \emph{indexed category} consists of a pseudofunctor $\underline{A} \colon \mathcal S^\text{op} \to \Cat$ such that each $\underline{A}^I := \underline{A}(I)$ has a distinguished class of isomorphisms called \emph{canonical}. The coherence conditions for $\underline{A}$ to be a pseudofunctor are specified in terms of these canonical isomorphisms, and an entire theory of indexed categories can be developed with respect to the `indexing base' $\mathcal S$ (there is a 2-category $\mathcal S\text{-Ind}$ of such structures, there is a notion of indexed adjunction, etc.\dots).

The theories of indexed categories and of fibrations over $\Cat$ are `essentially equivalent' in the sense that the Grothendieck construction allows to pass back and forth between the two notions.

See \cite{johnstone1978indexed} for a more detailed account of the theory of indexed categories, and in particular \cite{pare1978abstract} for an introduction on the subject; the notion of indexed category is almost always referred to the case where the indexing base is an elementary topos.
\end{rmk}

Fibrations enjoy similar stability properties to their topological analogs.

\begin{thm}\label{catstabthm} Fibrations are closed under composition and pullback along arbitrary functors.
\end{thm}
\begin{proof} See \cite[3.1]{grayfibred}.
\end{proof}

\begin{rmk} 
There is no analog of Proposition \ref{thm:dfib-ofs} whose right class is the class of fibrations for the rather pedestrian reason that the fibrations introduced in Definition \ref{defn:1-fib} are not closed under retracts. There is however an \emph{algebraic} weak factorization system whose \emph{algebraic} right maps are cloven fibrations. See \cite[Example 29(ii)]{BourkeGarner:AWFS1}.
\end{rmk}

Before giving examples, we mention the dual notion. 
\begin{defn}
 A functor $p\colon E \ra B$ is an \emph{opfibration} if $p \colon E^{\op} \ra B^{\op}$ is a fibration. A functor $p \colon E \ra B$ that is both a fibration and an opfibration is called a \emph{bifibration}. 
\end{defn}
The proof of the following lemma is left as an exercise.
\begin{lem}\label{lem:bifib-adjoint}
 A fibration $p \colon E \ra B$ is also an opfibration if and only if each functor $f^* \colon E_b \ra E_{b'}$ has a left adjoint $f_!$.
\end{lem}

Finally, some examples:

\begin{ex} \quad
\begin{enumroman} 
\item The codomain functor $C^{\2} \ra C$ is an opfibration that is a fibration iff $C$ has pullbacks. (Hence the name ``cartesian''.)
\item The domain functor $C^\2 \ra C$ is a fibration that is an opfibration iff $C$ has pushouts.
\item There is a functor $\Ring^\op \to \Cat$ which carries a ring $R$ to the category ${}_R\Mod$ of left $R$-modules and a ring homomorphism $f \colon R \to S$, to the restriction of scalars functor $f^* \colon {}_S\Mod \to {}_R\Mod$. As these functors admit left adjoints $f_!$, given by extension of scalars, by Lemma \ref{lem:bifib-adjoint} the Grothendieck construction produces a bifibration $\Mod \to \Ring$, where the objects of $\Mod$ are left modules over varying rings and a morphism over a ring homomorphism $f\colon R \to S$ can be understood as a homomorphism of $R$-modules, where the codomain $S$-module is restricted along $f$.
\item For any category $C$, the category of set-indexed families of objects of $C$ is a fibration over $\Set$ with the forgetful functor taking a family to its indexing set. The functors $f^*$ are given by reindexing and have left adjoints iff $C$ has small coproducts, and right adjoints iff $C$ has small products.\end{enumroman}
\end{ex}

We mention one final result which will motivate the definitions in Section \ref{2catsec}.

\begin{thm} A functor $p \colon E \ra B$ is a fibration if and only if the functor $[X,p] \colon [X,E] \ra [X,B]$ is a fibration for every category $X$.
\end{thm}
\begin{proof}
See \cite[3.6]{grayfibred}.
\end{proof}

\subsection{Two-sided discrete fibrations} Finally, we reach the variant of interest.

\begin{defn}\label{cattwodefn} 
A \emph{two-sided discrete fibration} is a span $A \sr{q}{\la} E \sr{p}{\ra} B$ such that 
\begin{enumroman} 
\item each $qe \ra a'$ in $A$ has a unique lift in $E$ that has domain $e$ and lies in the fiber over $pe$
\item each $b' \ra pe$ in $B$ has a unique lift in $E$ that has codomain $e$ and lies in the fiber over $qe$ 
\item for each $f \colon e \ra e'$ in $E$ the codomain of the lift of $qf$ equals the domain of the lift of $pf$ and their composite is $f$. 
\end{enumroman}
The situation is depicted in Figure \ref{fig:fibrs}, where in the lower right corner the ``vertical-cartesian'' factorization of an arrow $f$ is depicted.%\eQ{I had to comment this out because you didn't send me the figure but please add it back.}
\end{defn}

Let $\DFib(A,B)$ denote the full subcategory of $\Span(A,B)$ on the two-sided discrete fibrations. 
\begin{figure}
\begin{tikzpicture}[scale=.6]
\begin{scope}
\fill[gray!15] (0,0) ellipse (3cm and 2cm);
\fill (1.5,1) circle (2pt) node (vp0) {};
\fill (1.5,-.25) circle (2pt) node (vp1) {};
\fill (0,-1) circle (2pt) node (vq1) {};
\fill (-1.25,-1) circle (2pt) node (vq0) {};
\draw[->] (vp0) -- (vp1);
\draw[->] (vq0) -- (vq1);
\node[draw,dashed,fit=(vp0) (vp1)] (U) {};
\node[draw,dashed,fit=(vq0) (vq1)] (V) {};
\node[black!75] at (-.75,.75) {\huge$\boldsymbol E$};
\end{scope}
\begin{scope}[xshift=9cm]
\fill[gray!15] (0,0) ellipse (3cm and 2cm);
\fill (-1.5,1) circle (2pt) node (vp0) {};
\fill (-1.5,-.25) circle (2pt) node[below] (vp1) {$p$};
%\node[black] (vp1) at (-1.5,-.25) {$p\,\bullet$};
%\node[black] (vp1) at (-1.5,-.25) {$p\,\bullet$};
\fill (0,-1) circle (2pt) node (rightid) {};
\draw[->, shorten >=3pt] (vp0) -- (vp1);
\node[draw,dashed,fit=(vp0) (vp1)] (A) {};
\node[draw,dashed,fit=(rightid)] (B) {};
\node[black!75] at (.75,.75) {\huge$\boldsymbol B$};
\end{scope}
\begin{scope}[yshift=-7cm]
\fill[gray!15] (0,0) ellipse (3cm and 2cm);
\fill (0,1) circle (2pt) node (vq1) {};
\fill (-1.25,1) circle (2pt) node[left] (vq0) {$q$};
\fill (1.5,-.5) circle (2pt) node (lowid) {};
\draw[->, shorten <=4pt] (vq0) -- (vq1);
\node[draw,dashed,fit=(vq0) (vq1)] (C) {};
\node[draw,dashed,fit=(lowid),red] (D) {};
\node[black!75] at (-.75,-.75) {\huge$\boldsymbol A$};
\end{scope}
\draw[dashed,->,-latex,thick,blue] (U) -- (U |- D.north) node[midway,left] {$q$};
\draw[dashed,->,-latex,thick,red] (U) -- (U -| A.west) node[midway,above] {$p$};
\draw[dashed,->,-latex,thick,red] (V) -- (V -| B.west) node[midway,above] {$p$};
\draw[dashed,->,-latex,thick,blue] (V) -- (V |- C.north) node[midway,left] {$q$}; 
\begin{scope}[xshift=6cm,yshift=-5cm]
\fill (0,0) circle (2pt) node (f0) {};
\fill (1,-1) circle (2pt) node (f1) {};
\fill (f0 -| f1) circle (2pt) node (ff) {};
\draw[->] (f0) -- (ff);
\draw[->] (ff) -- (f1);
\end{scope}
\begin{scope}[xshift=11cm,yshift=-5cm]
\fill (0,0) circle (2pt) node (pf0) {};
\fill (1,-1) circle (2pt) node (pf1) {};
\fill (pf0 -| pf1) circle (2pt) node (pff) {};
\draw[double] (pf0) -- (pff);
\draw[->] (pff) -- (pf1);
\end{scope}
\begin{scope}[xshift=6cm,yshift=-8cm]
\fill (0,0) circle (2pt) node (qf0) {};
\fill (1,-1) circle (2pt) node (qf1) {};
\fill (qf0 -| qf1) circle (2pt) node (qff) {};
\draw[->] (qf0) -- (qff);
\draw[double] (qff) -- (qf1);
\end{scope}
\node[draw,dashed,fit=(f0) (ff),blue] (box1) {};
\node[draw,dashed,fit=(f1) (ff),red] (box2) {};
\node[draw,dashed,fit=(pf0) (pff),blue] (pbox1) {};
\node[draw,dashed,fit=(pf1) (pff),red] (pbox2) {};
\node[draw,dashed,fit=(qf0) (qff),blue] (qbox1) {};
\node[draw,dashed,fit=(qf1) (qff),red] (qbox2) {};
\draw[dashed,thick,->,-latex,red] (box1) -- (pbox1);
\draw[dashed,thick,->,-latex,blue] (box1) -- (qbox1);
\draw[dashed,thick,->,-latex,red] (box2) -- (pbox2);
\draw[dashed,thick,->,-latex,blue] (box2) -- (qbox2);
\end{tikzpicture}
\caption{The cartesian-vertical factorization of a morphism in a span: a fibration determines a factorization system on $E$ in the sense that every $f : e\to e'$ in $E$ admits a factorization as a $p$-vertical arrow followed by a $p$-cartesian arrow; dually, an opfibration determines a $q$-vertical, $q$-cocartesian factorization.}
\label{fig:fibrs}
\end{figure}
The definition of two-sided discrete fibration in fact doesn't add something new to the picture, as
\begin{thm} There exist equivalences of categories \[\DFib(A,B) \simeq [B^\op \times A, \Set] \simeq \DFib(B\times A^\op),\] pseudo-natural in $A$ and $B$.
\end{thm}
\begin{proof} Given a two-sided discrete fibration $A \sr{q}{\la} E \sr{p}{\ra} B$, define $B^\op \times A \ra \Set$ by $(b,a) \mapsto E_{a,b}$, the objects in the fiber over $a$ and $b$. Given $g \colon b' \ra b$, the corresponding function $g^* \colon E_{a,b} \ra E_{a,b'}$ sends $e \in E_{a,b}$ to the domain of the unique lift of $g$ in the fiber over $a$ with codomain $e$; likewise, given $f \colon a \ra a'$ the corresponding $f_* \colon E_{a,b} \ra E_{a',b}$ sends $e$ to the codomain of the  unique lift of $f$ in the fiber over $b$ with domain $a'$.

Conversely, given $P\colon B^\op \times A \ra \Set$, let the objects of $E$ be triples $(b \in B, e \in P(b,a), a \in A, )$ and morphisms $(b, e, a) \ra (b',e', a')$ be pairs of arrows $f\colon a \ra a'$ in $A$ and $g\colon b \ra b'$ in $B$ such that $f_*(e) = g^*(e')$.\footnote{Note this isn't the \emph{collage} of $P$, a category living over the walking arrow $\2$, defined below. Rather it's the category of sections of this functor, with morphisms the natural transformations. See the $n$Lab discussion of two-sided fibrations.}
\end{proof}

Comma categories provide an important class of examples of two-sided discrete fibrations. In fact, in $\Cat$, they tell the whole story.

\begin{thm}\label{catcommathm} For any opspan $B \sr{f}{\ra} C \sr{g}{\la} A$, its comma category \[\xymatrix{ \comma{f}{g} \ar[d]_{c} \ar[r]^d & B \ar[d]^f \\ A \ar[r]_g \ar@{}[ur]|{\Leftarrow} & C}\] is a two-sided discrete fibration $A \sr{c}{\la} \comma{f}{g} \sr{d}{\ra} B$. Furthermore, all two-sided discrete fibrations in $\Cat$ arise this way.
\end{thm}
\begin{proof} See \cite[14]{streetfibrationsyoneda}. 
\end{proof}

Finally, for completeness, we give the definition of two-sided fibrations, which aren't required to be discrete.

\begin{defn} A span $A \sr{q}{\la} E \sr{p}{\ra} B$ is a \emph{two-sided fibration} if \begin{enumroman} \item Any $g\colon qe \ra a \in A$ has an opcartesian lift with domain $e$ that lies in the fiber over the identity at $pe$. \item Any $f \colon b \ra pe$ has a cartesian lift with codomain $e$ that lies in the fiber over the identity at $qe$. \item Given a cartesian lift $f^*e \ra e$ of $f$ and an opcartesian lift $e \ra g_!e$ of $g$, as above, the composite \[ f^*e \ra e \ra g_!e\] lies over both $f$ and $g$. If $f^*e \ra g_!f^*e$ and $f^*g_!e \ra g_!e$ denote its opcartesian and cartesian lifts, then the canonical comparison $g_!f^*e \ra f^*g_!e$ induced by the universal property of either of these must be an isomorphism. \end{enumroman}
\end{defn}

Two-sided fibrations determine pseudo-functors $\B^\op \times \A \rightsquigarrow \Cat$.
\section{Fibrations in 2-categories}\label{2catsec}

The notions of fibration and two-sided discrete fibration internal to a 2-category are due to \cite{streetfibrationsyoneda}; a good summary of the main results can be found in \cite[\S 2]{weberyoneda}. In order to perform desired constructions, we work in \emph{finitely complete} 2-categories $\K$, i.e., a 2-category that admits finite conical limits and cotensors with the ``walking arrow'' category $\2$ --- though the weaker hypothesis that $\K$ admits only PIE-limits would also suffice. In particular:

\begin{lem} A finitely complete 2-category $\K$ has all comma objects.
\end{lem}
\begin{proof}
First note that
\[
\xymatrix@=10pt{ & A^\mathbb{2} \ar[dl]_c \ar[dr]^d \ar@{}[dd]|{\Leftarrow} \\ A \ar@{=}[dr] & & A \ar@{=}[dl] \\ & A}
\] is a comma object, where $d$ and $c$ are induced by the domain and codomain inclusions $\mathbb{1} \rightrightarrows \2$. To see this, recall that comma objects, like all weighted limits, are defined \emph{representably}, meaning in this case that the comma object $\comma{A}{A}$ of the depicted opspan must induce isomorphisms of categories \[ \K(X,\comma{A}{A}) \cong \comma{\K(X,A)}{\K(X,A)}\] for all $X \in \K$, where the right hand side denotes the comma category for the pair of identity functors on $\K(X,A)$. But we know that in $\Cat$, this comma category is $\K(X,A)^\2$. Hence, $\comma{A}{A}$ must induce isomorphisms of categories \[\K(X,\comma{A}{A}) \cong \K(X,A)^{\2}\]  which is the defining universal property of the cotensor of $A \in \K$ by $\2$.

Given an opspan $A \sr{g}{\ra} C \sr{f}{\la} B$ in $\K$, its comma object is the wide 2-pullback 
\[
\xymatrix@=10pt{
	& A \ar[dr]^g\ar@{=}[dl]&& C^{\2} \ar[dr]^d\ar[dl]_c&& B\ar@{=}[dr]\ar[dl]_f \\
	A && C && C && B.
}
\]
with 2-cell defined by whiskering the 2-cell of the comma object $C^\2$. This can be proven directly in $\Cat$, implying the result for a generic 2-category $\K$ by the representability of weighted limits.
\end{proof}

Another proof of the previous lemma, left as an exercise to the reader, uses the pasting lemma for comma squares.

\begin{lem}\label{pastinglem} Given a diagram in a 2-category $\K$ such that the right-hand square is a comma square \[\xymatrix{ \cdot \ar[d] \ar[r] & \cdot \ar[d] \ar[r] & \cdot \ar[d] \\ \cdot \ar[r] & \cdot \ar[r] \ar@{}[ur]|{\Leftarrow} & \cdot}\] the whole diagram is a comma square if and only if the left-hand square is a 2-pullback.
\end{lem}
\begin{proof} Analogous to the pasting lemma for ordinary pullbacks.
\end{proof}

\subsection{Fibrations}

Fibrations in a 2-category are defined representably. 

\begin{defn} A 1-cell $p\colon E \ra B$ in a 2-category $\K$ is a \emph{fibration} iff $\K(X,p)$ is a fibration  for all $X \in \K$ and if \[\xymatrix{ \K(X,E) \ar[d]_{\K(X,p)} \ar[r]^{\K(x,E)} & \K(Y,E) \ar[d]^{\K(Y,p)} \\ \K(X,B) \ar[r]_{\K(x,B)} & \K(Y,B)}\] is a cartesian functor for all $x \colon Y \ra X$ in $\K$.
\end{defn}

Unpacking this definition, $p \colon E \ra B$ is a \emph{fibration} if every 2-cell \[\xymatrix{ X \ar[r]^e \ar[dr]_b & E \ar[d]^p \\ \ar@{}[ur]|(.7){\Rightarrow \beta } & B}\] has a $p$-cartesian lift $\a \colon e' \Ra e$ so that $p\a = \beta $. A 2-cell 
\[\xymatrix{ X \ar@/^/[rr]^{e'} \ar@{}[rr]|{\Downarrow \a} \ar@/_/[rr]_e & &E}\] is $p$-\emph{cartesian} when for all $x \colon Y \ra X$, $\a x$ is a $\K(Y,p)$-cartesian arrow in $\K(Y,E)$. This means that for all 2-cells \[\xymatrix{ Y \ar[rr]^{e''} \ar[dr]_x & \ar@{}[d]|{\Downarrow \xi}& E & Y \ar[d]_x \ar[rr]^{e''} &  \ar@{}[d]|{\Downarrow \gamma} & E \ar[d]^p \\ & X \ar[ur]_e &  & X \ar[r]_{e'} & E \ar[r]_p & B}\] such that $p\xi = p\a x \cdot \gamma$, then there is a unique 2-cell $\zeta\colon e'' \Ra e'x$ such that $\xi =  \a x \cdot \zeta$ and $p\zeta = \gamma$.  

Note this definition did not require any hypotheses on the 2-category $\K$, but to prove the results that follow whose statements refer to certain finite 2-limits in $\K$, we do require something like the hypothesis of finite completeness to guarantee that these exist.

\begin{thm} In any finitely complete 2-category $\K$ \begin{enumroman} \item the composite of fibrations is a fibration \item the pullback of a fibration is a fibration \end{enumroman}
\end{thm}
\begin{proof} Follows from Theorem \ref{catstabthm}.
\end{proof}

\begin{thm} Let $\K$ be a finitely complete 2-category, $p \colon E \ra B$ a 1-cell. The following are equivalent: \begin{enumroman} \item $p$ is a fibration. \item For all $b\colon X \ra B$, the map $i\colon X \times_B E \ra \comma{b}{p}$ has a right adjoint in $\K/X$ \[\xymatrix@R=10pt@C=10pt{ X \times_B E \ar@{-->}[dr]_i \ar@/_/[dddr] \ar@/^/[drrr] \\ & \comma{b}{p} \ar[rr] \ar[dd] & & E \ar[dd]^p \\  & & &  \\ & X \ar[rr]_b \ar@{}[uurr]|{\Rightarrow}  & & B}\]

 \item The map $E \ra \comma{B}{p}$ has a right adjoint in $\K/B$. \item the canonical arrow $E^\2 \ra \comma{B}{p}$ has a right adjoint right inverse. \end{enumroman}
\end{thm}
\begin{proof} 
(iii) is (ii) with $b = 1_B$. (iii) implies (ii) by the pasting lemma \ref{pastinglem}. Equivalence with (i) requires some cleverness; see \cite[2.7]{weberyoneda}. (i) $\Leftrightarrow$ (iv) is analogous to Theorem \ref{catadjthm}; see \cite[9]{streetfibrationsyoneda}.
\end{proof}

An \emph{opfibration} in $\K$ is a fibration in $\K^{\co}$. It follows from characterization (iii) above that any 2-functor between finitely complete 2-categories that preserves comma objects preserves fibrations and opfibrations. We briefly mention the very simplest examples.

\begin{ex} \quad
\begin{enumroman}
\item The fibrations internal to the 2-category $\Cat$ are exactly the fibrations introduced in Definition \ref{defn:1-fib}.
\item A fibration internal to the 2-category $\Cat_{/A}$ is a functor $p\colon E \ra B$ such that arrows $b \ra pe$ in the fiber over an identity in $A$ have $p$-cartesian lifts. If the functors $E \ra A$ and $B \ra A$ are fibrations in $\Cat$ and $p$ preserves cartesian arrows, then $p$ is a fibration in $\Cat$ if and only if it is a fibration in $\Cat_{/A}$. In general, the notion of fibration in $\Cat_{/A}$ is weaker than the notion of fibration in $\Cat$.
\end{enumroman}
\end{ex}

Discrete fibrations in a 2-category $\K$ with cotensors by $\2$ can either be defined representably or in analogy with Definition \ref{internaldefn} and these definitions are equivalent. 

\subsection{Two-sided discrete fibrations} In a 2-category $\K$, we write $\Span(\K)$ for the bicategory of spans in $\K$, whose objects and 1-cells coincide with those in the 1-category underlying $\K$, and whose 2-cells are defined fiberwise. If $\K$ has binary products, the hom-categories $\Span(\K)(A,B)$ are isomorphic to the comma categories $\K_{/A \times B}$; hence, they are actually 2-categories. %Furthermore, the composition is 2-functorial.

\begin{defn} A span $A \sr{q}{\la} E \sr{p}{\ra} B$ is a \emph{two-sided discrete fibration} if and only if it is representably so, i.e., if for all $X \in \K$, \[ \xymatrix{ \K(X,A) & \K(X,E) \ar[l]_-{\K(X,q)} \ar[r]^-{\K(X,p)} & \K(X,B)}\] is a two-sided discrete fibration.
\end{defn}

As in $\Cat$, comma objects provide examples of two-sided discrete fibrations.

\begin{thm} Given $f \colon A \ra C$ and $g \colon B \ra C$, the span $A \la \comma{f}{g} \ra B$ is a two-sided discrete fibration.
\end{thm}
\begin{proof} Because weighted limits are also defined representably, it suffices to prove when $\K= \Cat$. See Theorem \ref{catcommathm}.
\end{proof}

\begin{thm} If $A \sr{q}{\la} E \sr{p}{\ra} B$ is a \emph{two-sided discrete fibration}, then $p$ is a fibration and $q$ is an opfibration.
\end{thm}
\begin{proof} The proof is technical, but  at least the two parts are dual, by interpreting the two-sided discrete fibration in  $\K^{\co}$.
\end{proof}
\begin{rmk}\label{fibinbic}
In his original paper, Street defines fibrations, opfibrations, and two-sided discrete fibrations to be pseudo-algebras for certain 2-monads on the appropriate hom-2-category of $\Span(K)$. For instance, the 2-monad on $\Span(K)(A,B)$ for two-sided discrete fibrations sends a span $A \sr{q}{\la} E \sr{p}{\ra} B$ to the 2-pullback of
\[ \xymatrix@=10pt{ & A^\2 \ar[dl]_{d} \ar[dr]^c & & E \ar[dl]_q \ar[dr]^p & & B^\2 \ar[dl]_d \ar[dr]^{c} \\ A & & A & & B & & B}\] 
See \cite{streetfibrationsyoneda} for details.
\end{rmk}
\subsection{Yoneda lemma} Part of the motivation for defining two-sided discrete fibrations internally to a 2-category was to state and prove a Yoneda lemma in this context. While this is peripheral to our discussion, we nonetheless take a brief detour to give the statement.

\begin{thm} Let $\K$ be finitely complete 2-category, $A \sr{q}{\la} E \sr{p}{\ra} B$ a two-sided discrete fibration, and $f\colon A \ra B$ a 1-cell. The identity 2-cell $\id_f$ induces a canonical arrow $i \colon A \ra \comma{B}{f}$ from the 2-pullback of $f$ along the identity at $B$ to the comma object. Precomposition with $i$ induces a bijection between arrows of spans $\comma{B}{f} \ra E$ and arrows of spans $B \ra E$.
\end{thm}
\begin{proof}[Proof sketch] \eadd{Rather than reproduce the full proof from \cite[16]{streetfibrationsyoneda} or \cite[2.12]{weberyoneda} here, we sketch the main ideas to outline a nice exercise in categorical yoga.}

Unwinding the definition, the 1-cell $i$ is induced by the 2-universal property of the comma object:
\[
\begin{tikzcd}[sep=1.25em]
& A \arrow[ddr, bend left, "f"] \arrow[ddl, bend right, equals] \arrow[d, dashed, "i"] \\ & \comma{B}{f} \arrow[dr, "d"] \arrow[dl, "c"'] \arrow[dd, phantom, "\scriptstyle\Leftarrow"] \\ A \arrow[dr, "f"'] & & B \arrow[dl, equals] \\ & B
\end{tikzcd}
\qquad = \qquad \begin{tikzcd}[sep=1.25em] & A \arrow[dr, "f"] \arrow[dl, equals] \arrow[dd, phantom, "\scriptstyle=\ \id_f"] \\ A \arrow[dr, "f"'] & & B \arrow[dl, equals] \\ & B
 \end{tikzcd}
 \]
The Yoneda lemma asserts that restriction along $i$ induces a bijection between the maps of spans:
	\[ 
\left\{
\vcenter{
\xymatrix@R=3mm@C=3mm{
& A \ar@{.>}[dd]\ar@{=}[dl]\ar[dr]^f & \\
A && B \\
& E\ar[ur]_p\ar[ul]^q &
}}
\right\} 
\leftrightarrows
\left\{
\vcenter{
\xymatrix@R=3mm@C=3mm{
& \comma{B}{f} \ar@{.>}[dd]\ar[dl]_c\ar[dr]^d & \\
A && B \\
& E\ar[ur]_p\ar[ul]^q &
}}
\right\}
	\]

 The classical Yoneda lemma is what we get from this theorem when $\K=\Cat$ and $A=1$ is the terminal category; the other leg $p\colon E \to B$ is then a discrete fibration. In this case,  the above correspondence reduces to
	\[
	\left\{
\vcenter{
\xymatrix@R=4mm@C=4mm{
1 \ar[dr]_b \ar[rr] && E \ar[dl]^p\\
&B &
}}\right\}
\qquad
\leftrightarrows
\qquad
\left\{
\vcenter{
\xymatrix@R=4mm@C=4mm{
\comma{B}{b} \ar[dr]_d\ar[rr] && E \ar[dl]^p\\
&B  &
}}\right\}
	\]
	Note that the left-hand side is isomorphic to the fiber $E_b$. Without loss of generality $p\colon E \to B$ can be thought of $\Sigma\colon\comma{*}{F} \to B$ for some $F \colon B^\op\to \Set$. Now via Theorem \ref{thm:discrete-equiv},  the Yoneda lemma asserts that there is a bijection between functors $\comma{B}{b} \to E$ over $B$ and elements of the set $Fb$.
\end{proof}

\section{Fibrations in bicategories}\label{bicatsec}

The notions of fibration and two-sided discrete fibration internal to a bicategory are due to  \cite{streetfibrationsbicategories}; a good summary of the main results can be found in \cite{carbonijohnsonstreetveritymodulated}. The first two sections are somewhat abbreviated; we excuse this laxity by mentioning that it enables us to quickly get to the main point in the final two sections. The reader who wishes to see statements analogous to those of Section \ref{2catsec} is encouraged to prove them, replacing any 2-limits that appear with the appropriate bilimits. 

Section \ref{codiscsec} relies heavily on the ``codiscrete cofibration'' entry at the $n$Lab.

\subsection{Fibrations}

In a bicategory, it is generally considered unreasonable to ask for an equality of 1-cells, but there is no moral objection to asking 2-cells to be equal. Thus, when defining fibrations internally to a generic bicategory $\K$, we can use the definition of $p$-cartesian 2-cells that was ``unpacked'' above, enabling the definition:

\begin{defn} A 1-cell $p \colon E \ra B$ in a bicategory $\K$ is a \emph{fibration} if for all 1-cells $e \colon X \ra E$ and 2-cells $\a \colon b \Ra pe \colon X \ra B$, there exists a $p$-cartesian $\chi \colon e' \Ra e$ for which there is an isomorphism $b \cong p e'$ whose composite with $p\chi$ is $\a$.
\end{defn}

\begin{ex} The fibrations internal to $\Cat$ as a bicategory are sometimes called \emph{Street fibrations}. Explicitly, a functor $p \colon E \ra B$ is a Street fibration if for every $f \colon b \ra pe$ in $B$, there is a $p$-cartesian arrow $g \colon e' \ra e$ and an isomorphism $h \colon b \cong pe'$ such that $f = pg \cdot h$.
\end{ex}

This notion of fibration is invariant under equivalence of categories. In particular, equivalences of categories are Street fibrations, though they are not necessarily fibrations in the classical sense.

\begin{lem} A 1-cell $p \colon E \ra B$ in a bicategory $\K$ is a fibration if and only if both \begin{enumroman}  \item for all $X \in \K$, $\K(X,p) \colon \K(X,E) \ra \K(X,B)$ is a Street fibration and \item for all 1-cells $x \colon Y \ra X$ in $\K$, precomposition with $x$ induces a map of fibrations $\K(X,p) \ra \K(Y,p)$. \end{enumroman}
\end{lem}

\subsection{Two-sided fibrations and two-sided discrete fibrations} First, we should say a few words about the tricategory $\Span(\K)$. When $\K$ is a bicategory, not a 2-category, we define the 1-cells and 2-cells of the bicategory $\Span(\K)(A,B)$ slightly differently. A morphism of spans from $A$ to $B$ is given by a 1-cell $f$ in $\K$ and isomorphic 2-cells as depicted \[\xymatrix{ & E \ar[dl]_q \ar[dr]^p \ar[dd]^f & \\ A \ar@{}[r]|{\mu \cong} & & B \ar@{}[l]|{\cong\nu} \\ & E' \ar[ul]^{q'} \ar[ur]_{p'} & }\] A 2-cell in $\Span(\K)(A,B)$ is a 2-cell $\theta\colon f \Ra f'$ that pastes together with one of each pair of 2-cell isomorphisms to give the other.

\begin{defn} A span $A \sr{q}{\la} E \sr{p}{\ra} B$ in $\K$ is a two-sided discrete fibration if \begin{enumroman} \item for every $e\colon X \ra E$ and 2-cell $\a\colon qe \Ra a \colon X \ra A$, there exists an opcartesian 2-cell $\chi \colon e \Ra e'$ and isomorphism $qe' \cong a$ whose composite with $q\chi$ is $\a$ and such that $p\chi$ is an isomorphism. \item for every $e\colon X \ra E$ and 2-cell $\beta  \colon b \Ra pe \colon X \ra B$, there exists a cartesian 2-cell $\zeta\colon e' \Ra e$ and isomorphism $b \cong pe'$ whose composite with $p\zeta$ is $\beta$ and such that $q\zeta$ is an isomorphism. \item for all $\eta, \eta' \colon e \Ra e' \colon X \ra E$, if $p\eta = p \eta'$, $q\eta = q\eta'$, and $p\eta$ and $q\eta$ are invertible, then $\eta = \eta'$ and is invertible. \end{enumroman}
\end{defn}

Condition (iii) is equivalent to saying that the span is representably essentially discrete, i.e., for all spans $E'$ from $A$ to $B$, the hom-category \[\Span(\K)(A,B) (E', E)\] is equivalent to a discrete category.

Proof of the following alternate characterization, which is due to \cite{carbonijohnsonstreetveritymodulated} and should be compared with Definition \ref{cattwodefn},  is left as an exercise.

\begin{lem} A span $A \sr{q}{\la} E \sr{p}{\ra} B$ is a two-sided discrete fibration if and only if the following conditions hold. \begin{enumroman} \item for all arrows $e \colon X \ra E$ and 2-cells $\a \colon qe \Ra a$, the category whose objects are pairs $(\chi\colon e \Ra e', \nu \colon qe' \cong a)$ with $\a = \nu \cdot q\chi$ and $p\chi$ invertible is essentially discrete and non-empty; \item for all arrows $e \colon X \ra E$ and 2-cells $\beta  \colon b \Ra pe$, the category whose objects are pairs $(\zeta \colon e' \Ra e, \mu \colon b \cong pe')$ with $\beta = p\zeta \cdot \mu$ and $q\zeta$ invertible is essentially discrete and non-empty; \item each 2-cell $\eta \colon e \Ra e' \colon X \ra E$ is a composite $\zeta \chi$ where $p\chi$ and $q\zeta$ are invertible. \end{enumroman} \end{lem}

By a comma object in a bicategory, we mean the bilimit with the shape described above, relaxing the defining isomorphism \[\K(X,\comma{f}{g}) \simeq \K(X,f)/\K(X,g)\] of categories to an equivalence.

\begin{thm} Any comma object in a bicategory gives a two-sided discrete fibration.
\end{thm}
\begin{proof} See \cite[3.44]{streetfibrationsbicategories}.
\end{proof}

\subsection{Two-sided codiscrete cofibrations}\label{codiscsec}

For this section, the motivating example is the 2-category $\K=\V$-$\Cat$ of categories enriched in some complete and cocomplete closed symmetric monoidal category $(\V, \otimes, I)$. We'll see in the next section what is special about the case $\V=\Set$, $\K = \Cat$.

In enriched category theory, $\V$-profunctors play an important role; if $\A,\B \in \V$-$\Cat$, a $\V$-\emph{profunctor} from $\A$ to $\B$ is a $\V$-functor $\B^\op \otimes \A \ra \V$. A warning: unless $\V$ is cartesian monoidal, the tensor product of $\V$-categories is distinct from their cartesian product. The tensor product of $\V$-categories gives the morally correct notion of $\V$-profunctors and is necessary for the construction of collages below.

We would like to be able to model $\V$-profunctors internally to the 2-category of $\V$-categories because this will make it easier to understand which pseudo-functors $\V$-$\Cat \rightsquigarrow \K$ ``preserve'' profunctors. One way to describe the data of a $\V$-profunctor in $\V$-$\Cat$ is through its collage.

\begin{defn} The \emph{collage} of $F \colon \B^\op \otimes \A \ra \V$ is a opspan $\A \ra \E \la \B$, where $\E$ is the $\V$-category with objects $\ob\A \sqcup \ob \B$ and hom-objects \[ \E(b',b) = \B(b',b), \quad \E(b,a) = F(b,a), \quad \E(a,a') = \A(a,a'), \quad \E(a,b) = \emptyset,\] for all $a,a' \in \A$ and $b,b' \in \B$. The $\V$-functors $\A \ra \E$, $\B \ra \E$ are the inclusions.
\end{defn}

The main result is the following theorem of \cite{streetfibrationsbicategories}:

\begin{thm} The collages for $\V$-profunctors are exactly the two-sided codiscrete cofibrations in $\V$-$\Cat$, regarded as a bicategory.
\end{thm}

The reader may have already guessed the following definitions.

\begin{defn}
A opspan $A \ra E \la B$ in a bicategory $\K$ is a \emph{two-sided cofibration} if and only if it is a two-sided fibration in $\K^\op$, the bicategory with 1-cells reversed. The span is \emph{codiscrete} if it is representably discrete in $\Opspan(\K)(A,B) \cong {}^{A \sqcup B/}\K$, that is, if for all opspan $A \ra E' \la B$ the hom-category $\Opspan(\K)(A,B)(E,E')$ is equivalent to a discrete category.
\end{defn}

To give a complete encoding of profunctors from $A$ to $B$ internally to the bicategory $\K$, we need to be able to compose a two-sided codiscrete cofibration from $A$ to $B$ and from $B$ to $C$ and obtain a two-sided codiscrete cofibration from $A$ to $C$. If we removed the word ``codiscrete,'' this would be a piece of cake. So long as $\K$ has finite bicolimits, cofibrations are stable under pushout and composition. Hence, the pushout-composite of a opspan from $A$ to $B$ and a opspan from $B$ to $C$ is a opspan from $A$ to $C$ that is a two-sided cofibration if the original opspans were. This composition law is associative up to isomorphism, which is good enough. 

However, the resulting two-sided cofibration is unlikely to be codiscrete, whether or not the original two-sided cofibrations were. For instance, given $\V$-profunctors $\B^\op \otimes \A \ra \V$ and $\C^\op \otimes \B \ra \V$ and considering their collages, the pushout $\A \ra \E \la \C$ is a $\V$-category with objects $\ob \A \sqcup \ob \B \sqcup \ob \C$ called a \emph{gamut}; because of the presence of objects of $\B$, this is too fat to be a collage for a $\V$-profunctor $\C^\op \otimes \A \ra \V$. 

This problem can be solved provided there is a method for coreflecting from two-sided cofibrations into two-sided codiscrete cofibrations; a subcategory is \emph{coreflective} if the inclusion has a right adjoint.  In some examples, there may be a limit construction that achieves this. This is the approach that Street takes originally, but see \cite{streetfibrationscorrection}.

A simpler approach is to ask that $\K$ have an orthogonal factorization system whose left class is generated by the two-sided codiscrete cofibrations $A \sqcup B \ra E$. An \emph{orthogonal factorization system} in a bicategory consists of two classes $(\mathcal{E},\mathcal{M})$ of 1-cells such that \begin{enumroman} \item Every 1-cell in $\K$ is isomorphic to the composite of a 1-cell in $\mathcal{E}$ followed by a 1-cell in $\mathcal{M}$ \item For all $e \colon X \ra Y \in \mathcal{E}$, $m \colon Z \ra W \in \mathcal{M}$, the square \[ \xymatrix{\K(Y,Z) \ar[d]_{\K(e,Z)} \ar[r]^{\K(Y,m)} & \K(Y,W) \ar[d]^{\K(e,W)} \\ \K(X,Z) \ar[r]_{\K(X,m)} \ar@{}[ur]|{\cong} & \K(X,W) }\] is a bipullback in $\Cat$.
\end{enumroman}

An orthogonal factorization system $(\mathcal{E},\mathcal{M})$ is \emph{generated} by a collection of 1-cells if the right class consists of precisely those 1-cells that satisfy axiom (ii) for all $e$ in this collection. When the generators are taken to be the codiscrete cofibrations, arrows in the right class are necessarily representably fully faithful. If the right class is stable under pushout and cotensor with $\2$, then the composite of a pair of two-sided codiscrete cofibrations can be defined by factoring the opspan $A \sqcup C \ra E$ formed by taking their pushout. This is the approach of \cite{carbonijohnsonstreetveritymodulated} and the $n$Lab.

We record this fact in the following theorem.

\begin{thm} Suppose $\K$ is a bicategory with finite limits and colimits. If the two-sided codiscrete cofibrations $A \sqcup B \ra E$ generate an orthogonal factorization system whose right class is closed under pushout and cotensor with $\2$, then there is a bicategory $\DCof(\K)$ whose objects are the objects of $\K$, whose 1-cells $A \ra B$ are the two-sided codiscrete cofibrations from $A$ to $B$, and whose 2-cells are isomorphism classes of morphisms of opspans.
\end{thm}
\begin{proof} See \cite[4.20]{carbonijohnsonstreetveritymodulated}.
\end{proof}

Here is how this works in our main example.

\begin{lem} $\V$-$\Cat$ has an orthogonal factorization system whose left class consists of the essentially surjective $\V$-functors and whose right class consists of the $\V$-fully faithful functors that is generated by the two-sided codiscrete cofibrations. 
\end{lem}
\begin{proof}
We leave it to the reader to prove that this orthogonal factorization exists; we show that it is generated by the two-sided codiscrete cofibrations. The collages are surjective on objects, so $\V$-fully faithful functors are necessarily right orthogonal to them. It remains to show that any $\V$-functor $F \colon \C \ra \mathcal{D}$ right orthogonal to the collages $\A \sqcup \B \ra \E$ is necessarily $\V$-fully faithful. Let $\mathcal{I}$ denote the $\V$-category with one-object and the unit as its hom-object. A $\V$-profunctor from $\mathcal{I}$ to itself is specified by a single object in $\V$. Given $c,c' \in \C$, form the collage of the $\V$-profunctor $\mathcal{I}^\op \otimes \mathcal{I} \ra \V$ determined by $\mathcal{D}(Fc,Fc')$. This collage has the form $\mathcal{I} \sqcup \mathcal{I} \ra \E$, where $\E$ has two objects 0,1 and one non-trivial hom $\E(0,1) = \mathcal{D}(Fc,Fc')$. The lifting problem whose bottom edge is the $\V$-functor that maps surjectively onto the hom-object $\mathcal{D}(Fc,Fc')$ \[\xymatrix{ \mathcal{I} \sqcup \mathcal{I} \ar[r]^- {c \sqcup c'} \ar[d] & \C \ar[d]^F \\ \E \ar[r]  \ar@{-->}[ur] & \mathcal{D} }\] must have a unique solution, which shows that $F$ is $\V$-fully faithful.
\end{proof}

It remains to check that $\V$-fully faithful functors are stable under pushout and cotensors with $\2$; we leave this to the reader. Putting these results together, we obtain the following corollary.

\begin{cor} There is a bicategory whose objects are $\V$-categories, whose 1-cells are two-sided codiscrete cofibrations in $\V$-$\Cat$, and whose 2-cells are isomorphism classes of maps of opspans.
\end{cor}

This bicategory can be used to ``equip'' the 2-category of $\V$-categories with a bicategory of $\V$-profunctors \cite{streetfibrationsbicategories}.

\subsection{A final note on modeling profunctors in $\Cat$}

We now have two models for profunctors in $\Cat$, the two-sided discrete fibrations and the two-sided codiscrete fibrations. It turns out there is a formal reason that these are the same. 

In any 2-category $\K$ with comma and cocomma objects, there is an adjunction 
\[ \xymatrix{ \mathrm{cocomma} \colon \Span(\K)(A,B) \ar@<1ex>[r] \ar@{}[r]|-{\perp} & \Opspan(\K)(A,B) \colon \mathrm{comma} \ar@<1ex>[l]}\] We've seen above that comma objects are always two-sided discrete fibrations; dually, cocomma objects are always two-sided codiscrete cofibrations. In $\Cat$, this adjunction is \emph{idempotent} in the sense that the comma object of the cocomma object of a comma object is isomorphic to the original comma object; this is equivalent to the dual statement. Any such adjunction restricts to an adjoint equivalence between the full subcategories in the image of each functor, which are consequently reflective and coreflective subcategories of the originals. So this adjunction restricts to an equivalence between the reflective subcategory of two-sided discrete fibrations and the coreflective subcategory of two-sided codiscrete cofibrations. Hence, both of these are equivalent to the 2-category of profunctors from $A$ to $B$.  
\bibliography{biblio}{}

\newcommand{\etalchar}[1]{$^{#1}$}
\providecommand{\bysame}{\leavevmode\hbox to3em{\hrulefill}\thinspace}
\providecommand{\MR}{\relax\ifhmode\unskip\space\fi MR }
\providecommand{\MRhref}[2]{%
  \href{http://www.ams.org/mathscinet-getitem?mr=#1}{#2}
}
\providecommand{\href}[2]{#2}
\begin{thebibliography}{JPW{\etalchar{+}}78}

\bibitem[BG16]{BourkeGarner:AWFS1}
John Bourke and Richard Garner, \emph{Algebraic weak factorisation systems {I}:
  {A}ccessible {AWFS}}, J. Pure Appl. Algebra \textbf{220} (2016), no.~1,
  108--147. \MR{3393453}

\bibitem[BK17]{berger2017comprehensive}
Clemens Berger and Ralph~M Kaufmann, \emph{Comprehensive factorisation
  systems}, Tbilisi Mathematical Journal \textbf{10} (2017), no.~3, 255--277.

\bibitem[CJSV94]{carbonijohnsonstreetveritymodulated}
Aurelio Carboni, Scott Johnson, Ross Street, and Dominic Verity,
  \emph{Modulated bicategories}, J. Pure Appl. Algebra \textbf{94} (1994),
  no.~3, 229--282. \MR{1285544 (96f:18008)}

\bibitem[FGI{\etalchar{+}}05]{MR2222646}
Barbara Fantechi, Lothar G\"{o}ttsche, Luc Illusie, Steven~L. Kleiman, Nitin
  Nitsure, and Angelo Vistoli, \emph{Fundamental algebraic geometry},
  Mathematical Surveys and Monographs, vol. 123, American Mathematical Society,
  Providence, RI, 2005, Grothendieck's FGA explained. \MR{2222646}

\bibitem[Gra66]{grayfibred}
John~W. Gray, \emph{Fibred and cofibred categories}, Proc. {C}onf.
  {C}ategorical {A}lgebra ({L}a {J}olla, {C}alif., 1965), Springer, New York,
  1966, pp.~21--83. \MR{0213413 (35 \#4277)}

\bibitem[Gro61]{grothendieckcategoriesfibrees}
Alexander Grothendieck, \emph{Cat\'egories fibr\'ees et descente}, Seminaire de
  g\'eom\'etrie alg\'ebrique de l'Institute des Hautes \'Etudes Scientifiques,
  1961.

\bibitem[Gro95]{grothendieckdescenteI}
\bysame, \emph{Technique de descente et th\'eor\`emes d'existence en
  g\'eometrie alg\'ebrique. {I}. {G}\'en\'eralit\'es. {D}escente par morphismes
  fid\`element plats}, S\'eminaire {B}ourbaki, {V}ol.\ 5, Soc. Math. France,
  Paris, 1995, pp.~Exp.\ No.\ 190, 299--327. \MR{1603475}

\bibitem[Joh02]{MR1953060}
Peter~T. Johnstone, \emph{Sketches of an elephant: a topos theory compendium.
  {V}ol. 1}, Oxford Logic Guides, vol.~43, The Clarendon Press, Oxford
  University Press, New York, 2002. \MR{1953060}

\bibitem[JPW{\etalchar{+}}78]{johnstone1978indexed}
Peter~T Johnstone, Robert Par{\'e}, RJ~Wood, D~Schumacher, GC~Wraith, and
  RD~Rosebrugh, \emph{Indexed categories and their applications}, Springer,
  1978.

\bibitem[Kel89]{kelly1989elementary}
Gregory~Maxwell Kelly, \emph{Elementary observations on 2-categorical limits},
  Bulletin of the Australian Mathematical Society \textbf{39} (1989), no.~2,
  301--317.

\bibitem[Law70]{lawvere1970equality}
F~William Lawvere, \emph{Equality in hyperdoctrines and comprehension schema as
  an adjoint functor}, Applications of Categorical Algebra \textbf{17} (1970),
  1--14.

\bibitem[Lur09]{lurietopos}
Jacob Lurie, \emph{Higher topos theory}, Annals of Mathematics Studies, vol.
  170, Princeton University Press, Princeton, NJ, 2009. \MR{MR2522659}

\bibitem[PS78]{pare1978abstract}
Robert Par{\'e} and Dietmar Schumacher, \emph{Abstract families and the adjoint
  functor theorems}, Indexed categories and their applications, Springer, 1978,
  pp.~1--125.

\bibitem[RV19]{riehlverityelements}
Emily Riehl and Dominic Verity, \emph{Elements of $\infty$-category theory},
  draft textbook available from www.math.jhu.edu/$\sim$eriehl/elements.pdf,
  2019.

\bibitem[Str74]{streetfibrationsyoneda}
Ross Street, \emph{Fibrations and {Y}oneda's lemma in a {$2$}-category},
  Category {S}eminar ({P}roc. {S}em., {S}ydney, 1972/1973), Springer, Berlin,
  1974, pp.~104--133. Lecture Notes in Math., Vol. 420. \MR{0396723 (53 \#585)}

\bibitem[Str80]{streetfibrationsbicategories}
\bysame, \emph{Fibrations in bicategories}, Cahiers Topologie G\'eom.
  Diff\'erentielle \textbf{21} (1980), no.~2, 111--160. \MR{574662 (81f:18028)}

\bibitem[Str87]{streetfibrationscorrection}
\bysame, \emph{Correction to: ``{F}ibrations in bicategories'' [{C}ahiers
  {T}opologie {G}\'eom.\ {D}iff\'erentielle {\bf 21} (1980), no.\ 2, 111--160;
  {MR}0574662 (81f:18028)]}, Cahiers Topologie G\'eom. Diff\'erentielle
  Cat\'eg. \textbf{28} (1987), no.~1, 53--56. \MR{903151 (88i:18004)}

\bibitem[SW73]{MR0346027}
Ross Street and R.~F.~C. Walters, \emph{The comprehensive factorization of a
  functor}, Bull. Amer. Math. Soc. \textbf{79} (1973), 936--941. \MR{0346027}

\bibitem[Web07]{weberyoneda}
Mark Weber, \emph{Yoneda structures from 2-toposes}, Appl. Categ. Structures
  \textbf{15} (2007), no.~3, 259--323. \MR{2320763 (2008h:18001)}

\end{thebibliography}
\bibliographystyle{amsalpha}
\hrulefill 
\end{document}